\documentclass[12pt]{amsart}
\usepackage{amsmath,amssymb}
\usepackage{amsfonts}
\usepackage{amsthm}
\usepackage{latexsym}
\usepackage{graphicx}

\def\p{\partial}

\def\R{\mathbb{R}}

\def\vv<#1>{\langle#1\rangle}

\def\1{\mathbf{1}}

\def\Mod{{\rm mod}}

\def\XXint#1#2{\setbox0=\hbox{$#1{#2}{\int}$}{#2}\kern-.5\wd0 }

\def\XXint#1#2#3{{\setbox0=\hbox{$#1{#2#3}{\int}$}
     \vcenter{\hbox{$#2#3$}}\kern-.5\wd0}}

\def\vv<#1>{{\left\langle#1\right\rangle}}

\def\CD{{\rm CD}}
\def\Deg{{\rm Deg}}
\def\Lip{{\rm Lip}}
\newtheorem{thm}{Theorem}[section]

\newtheorem{cor}{Corollary}[section]
\theoremstyle{definition}

\theoremstyle{remark}

\newtheorem{rem}{Remark}[section]

\numberwithin{equation}{section}

\begin{document}
\title{Comparison of Steklov eigenvalues and Laplacian eigenvalues on graphs}

\author{Yongjie Shi$^1$}
\address{Department of Mathematics, Shantou University, Shantou, Guangdong, 515063, China}
\email{yjshi@stu.edu.cn}
\author{Chengjie Yu$^2$}
\address{Department of Mathematics, Shantou University, Shantou, Guangdong, 515063, China}
\email{cjyu@stu.edu.cn}
\thanks{$^1$Research partially supported by NNSF of China with contract no. 11701355. }
\thanks{$^2$Research partially supported by GDNSF with contract no. 2021A1515010264 and NNSF of China with contract no. 11571215.}
\renewcommand{\subjclassname}{%
  \textup{2010} Mathematics Subject Classification}
\subjclass[2010]{Primary 05C50; Secondary 39A12}
\date{}
\keywords{Steklov eigenvalue, Laplacian eigenvalue, eigenvalue comparison}
\begin{abstract}
In this paper, we obtain a comparison of Steklov eigenvalues and Laplacian eigenvalues on graphs and discuss its rigidity. As applications of the comparison of eigenvalues, we obtain  Lichnerowicz-type estimates  and some combinatorial estimates for Steklov eigenvalues on graphs.
\end{abstract}
\maketitle\markboth{Shi \& Yu}{Comparison of Steklov eigenvalues and Laplacian eigenvalues}
\section{Introduction}

On a compact Riemannian manifold with boundary, the Dirichlet-to-Neumann map or Steklov operator sends the Dirichlet boundary data of a harmonic function on the manifold to its Neumann boundary data. The eigenvalues of the Dirichlet-to-Neumann map or Steklov operator are called Steklov eigenvalues of the Riemannian manifold. This subject was first introduced by Steklov \cite{KK,ST} when considering liquid sloshing. It was later found useful in applied mathematics, especially in electrical impedance tomography for medical imaging (see \cite{UL}).

In the discrete setting, Dirichlet-to-Neumann maps and Steklov eigenvalues were introduced in \cite{HHW,HM} and received attentions recently (see \cite{HH,HHW2,Pe,Pe2,SY,SY2}). In this paper, we obtain a comparison of Steklov eigenvalues and Laplacian eigenvalues on graphs. It seems that this is a major difference of Steklov eigenvalues on graphs with that on Riemannian manifolds. Such a comparison was also mentioned in \cite{HHW2} for graphs equipped with normalized weights. One of the application of this comparison of eigenvalues is in obtaining Lichnerowicz-type estimate for Steklov eigenvalues which in some sense extend Escobar's conjecture (see \cite{ES}) from the smooth case to discrete case. Escobar's conjecture was recently partially solved by Xiong-Xia \cite{XX}.

Let's recall some notations and notions on graphs before stating our main results.

A weighted graph is a triple $(G,m,w)$ where $G$ is a graph, $m$ is the vertex measure which is a positive function on the set $V(G)$ of the vertices of $G$ and $w$ is the edge weight which is a positive function on the set $E(G)$ of the edges  of $G$. For convenience, we view $w$ as a symmetric function on $V\times V$ by zero extension and we simply write $m(x)$ and $w(x,y)$ as $m_x$ and $w_{xy}$ respectively. Throughout this paper, the graph $G$ is assumed to be finite, simple and connected. We will also simply write $V(G)$ and $E(G)$ as $V$ and $E$ if no confusion was made. For $A,B\subset V$, $E(A,B)$ denotes the set of edges in $G$ with one end point in $A$ and the  other in  $B$. We call the weight with $m\equiv 1$ and $w\equiv 1$ a unit weight. For each $x\in V$,
\begin{equation}
\Deg(x):=\frac{1}{m_x}\sum_{y\in V}w_{xy}
\end{equation}
is called the weighted degree of $x$. If for any $x\in V$, $\Deg(x)=1$, the graph is called a graph with a normalized unit weight.

Let $(G,m,w)$ be a weighted graph. Let $A^0(G)$ be the space of functions on $V$ and $A^1(G)$ be the space of skew-symmetric functions $\alpha$ on $V\times V$ such that $\alpha(x,y)=0$ when $x\not\sim y$. Equip $A^0(G)$ and $A^1(G)$ with the natural inner products
\begin{equation}
\vv<u,v>=\sum_{x\in V}u(x)v(x)m_x
\end{equation}
and
\begin{equation}
\vv<\alpha,\beta>=\sum_{\{x,y\}\in E}\alpha(x,y)\beta(x,y)w_{xy}=\frac12\sum_{x,y\in V}\alpha(x,y)\beta(x,y)w_{xy}
\end{equation}
respectively. For any $u\in A^0(G)$, define the differential $du$ of $u$ as
\begin{equation}
du(x,y)=\left\{\begin{array}{ll}u(y)-u(x)&\{x,y\}\in E\\0&\mbox{otherwise.}\end{array}\right.
\end{equation}
Let $d^*:A^1(G)\to A^0(G)$ be the adjoint operator of $d:A^0(G)\to A^1(G)$. The Laplacian operator on $A^0(G)$ is defined as
\begin{equation}
\Delta=-d^*d.
\end{equation}
By direct computation,
\begin{equation}
\Delta u(x)=\frac{1}{m_x}\sum_{y\in V}(u(y)-u(x))w_{xy}
\end{equation}
for any $x\in V$. Moreover, by the definition of $\Delta$, it is clear that
\begin{equation}\label{eq-integration-by-part}
\vv<\Delta u,v>=-\vv<du,dv>
\end{equation}
for any $u,v\in \R^V$. So $-\Delta$ is a nonnegative self-adjoint operator on $A^0(G)=\R^V$. Let
\begin{equation}
0=\mu_1<\mu_2\leq \cdots \leq \mu_{|V|}
\end{equation}
be the eigenvalues of $-\Delta$ on $(G,m,w)$. It is clear that $\mu_1=0$ because constant functions are the corresponding eigenfunctions and $\mu_2>0$ because we always assume that $G$ is connected.

Next  recall the notion of graphs with boundary.  A pair $(G,B)$ is said to be a graph with boundary if $G$ is graph and $B\subset V(G)$ such that (i) any two vertices in $B$ are not adjacent, (ii) any vertex in $B$ is adjacent to some vertex
in $\Omega:=V\setminus B$. The set $B$ is called the vertex-boundary of $(G,B)$ and the set $\Omega$ is called the vertex-interior of $(G,B)$. An edge joining a boundary vertex and an interior vertex is called a boundary edge. The induced graph of $G$ on $\Omega$ is denoted as $G|_\Omega$.

Let $(G,m,w,B)$ be a weighted graph with boundary. For any $u\in \R^V$ and $x\in B$, define the normal derivative of $u$ at $x$ as:
\begin{equation}\label{eq-normal-derivative}
\frac{\p u}{\p n}(x):=\frac{1}{m_x}\sum_{y\in V}(u(x)-u(y))w_{xy}=-\Delta u(x).
\end{equation}
Then, by \eqref{eq-integration-by-part}, one has the following Green's formula:
\begin{equation}\label{eq-Green}
\vv<\Delta u,v>_\Omega=-\vv<du,dv>+\vv<\frac{\p u}{\p n},v>_B.
\end{equation}
Here, for any set $S\subset V$,
\begin{equation}
\vv<u,v>_S:=\sum_{x\in S}u(x)v(x)m_x.
\end{equation}
Similarly, for any $S\subset E$,
\begin{equation}
\vv<\alpha,\beta>_S:=\sum_{\{x,y\}\in S}\alpha(x,y)\beta(x,y)w_{xy}.
\end{equation}

We are now ready to introduce the notions of Dirichlet-to-Neumann map and Steklov eigenvalues on graphs.   For each $f\in \R^B$, let $u_f$ be the harmonic extension of $f$ into $\Omega$:
\begin{equation}
\left\{\begin{array}{ll}\Delta u_f(x)=0&x\in\Omega\\
u_f(x)=f(x)&x\in B.
\end{array}\right.
\end{equation}
Define the Dirichlet-to-Neumann map $\Lambda:\R^B\to \R^B$ as
\begin{equation}
\Lambda(f)=\frac{\p u_f}{\p n}.
\end{equation}
By \eqref{eq-Green},
\begin{equation}
\vv<\Lambda(f),g>_B=\vv<du_f,du_g>
\end{equation}
for any $f,g\in \R^B$. This implies that $\Lambda$ is a nonnegative self-adjoint operator on $\R^B$. Let
\begin{equation}
0=\sigma_1<\sigma_2\leq\cdots\leq \sigma_{|B|}
\end{equation}
be the eigenvalues of $\Lambda$. It is clear that $\sigma_1=0$ because constant functions are the corresponding eigenfunctions and $\sigma_2>0$ because we always assume that $G$ is connected.

We are now ready to state the first main result of this paper, a comparison of the the Steklov eigenvalues and the Laplacian eigenvalues on graphs, and its rigidity.
\begin{thm}\label{thm-comparison}
Let $(G,m,w,B)$ be a connected weighted finite graph with boundary. Then,
\begin{equation}\label{eq-comparison}
\sigma_i\geq \mu_i
\end{equation}
for $i=1,2,\cdots, |B|$. If $\sigma_i=\mu_i$ for some $i=2,3,\cdots,|B|$, then there is an eigenfunction $v_i$ of $\mu_i$ such that $v_i|_\Omega=\Delta v_i|_\Omega=0$.  Moreover, the quality of \eqref{eq-comparison} holds for all $i=1,2,\cdots, |B|$ if and only if the following statements are true:

(1) There is a nonnegative function $\rho \in \R^\Omega$, such that
\begin{equation}\label{eq-w-split}
w_{xy}=\rho_ym_xm_y
\end{equation}
for each $x\in B$ and $y\in \Omega$. In particular, each interior vertex is either adjacent to all boundary vertices or adjacent to no boundary vertices. 

(2) For any $u\in\R^\Omega$ with $\vv<u,1>_\Omega=0$,
\begin{equation}\label{eq-nonnegative}
\vv<du,du>_\Omega-\Deg\vv<u,u>_\Omega+V_B\vv<\rho u,u>_\Omega-\frac{V_G}{\Deg}\vv<\rho,u>_\Omega^2\geq 0
\end{equation}
Here, $$\Deg:=\vv<\rho,1>_\Omega=\Deg(x)$$ for any $x\in B$,
 $V_B=\sum_{x\in B}m_x$, $V_\Omega=\sum_{y\in\Omega}m_y$, $V_G=V_B+V_\Omega$, and
 $$\vv<du,du>_\Omega:=\vv<du,du>_{E(\Omega,\Omega)}.$$
\end{thm}

The eigenvalue comparison \eqref{eq-comparison} is almost an obvious observation from definitions as follows. Note that the Laplacian eigenvalues $\mu_i$'s can be obtained by applying Courant's min-max principle to the Rayleigh quotient:
\begin{equation}
R[u]=\frac{\vv<du,du>}{\vv<u,u>}
\end{equation}
and the Steklov eigenvalues $\sigma_i$'s can be obtained by applying Courant's min-max principle to the Rayleigh quotient:
\begin{equation}
R_{\sigma}[u]=\frac{\vv<du,du>}{\vv<u,u>_B}.
\end{equation}
It is clear that
\begin{equation}
R_{\sigma}[u]\geq R[u].
\end{equation}
Then, the eigenvalue comparison \eqref{eq-comparison} follows directly from Courant's min-max principle. The comparison \eqref{eq-comparison} was also mentioned in \cite[Corollary 1.6]{HHW2} for normalized weighted graphs. Our contribution here is characterizing the rigidity of \eqref{eq-comparison}.

As a direct consequence of \eqref{eq-nonnegative}, we have the following sufficient condition for \eqref{eq-comparison} to hold for all $i=1,2,\cdots, |B|$ on general weighted graphs.
\begin{cor}\label{cor-general}
Let $(G,m,w,B)$ be a connected weighted finite graph, suppose that for any $x\in B$ and $y\in \Omega$, $w_{xy}=\rho_ym_x m_y$ for some nonnegative function $\rho\in \R^\Omega$, and
\begin{equation}\label{eq-cor-general}
\mu_2(\Omega)\geq \Deg-V_B\rho_{\min}+\frac{V_G}{\Deg}\vv<\rho,\rho>_\Omega
\end{equation}
where $\mu_2(\Omega)$ is the second Laplacian eigenvalue of the induced graph of $G$ on $\Omega$, $$\Deg:=\vv<\rho,1>_\Omega=\Deg(x)$$ for any $x\in B$, and $\rho_{\min}=\min_{y\in\Omega}\rho_y$. Then, $\sigma_i=\mu_i$ for all $i=1,2,\cdots, |B|$.
\end{cor}

By using Corollary \ref{cor-general}, one can construct many graphs such that equality of \eqref{eq-comparison} holds for $i=1,2,\cdots, |B|$. For example, first fix a connected weighted graph $(\Omega, m_\Omega, w_\Omega)$. Then add the boundary $B$ to $\Omega$, joining $B$ to $\Omega$ and arranged the measure of the boundary vertices and weight of the boundary edges so that $$w_{xy}=m_xm_y\rho_y$$ holds for any $x\in B$ and $y\in \Omega$. Note that $\mu_2(G|_\Omega,m_\Omega,w_\Omega)>0$ and
$$\mu_2(G|_\Omega,m_\Omega,\lambda w_\Omega)=\lambda\mu_2(G|_\Omega,m_\Omega,w_\Omega).$$
So, if $\lambda$ is large enough, the equality \eqref{eq-cor-general} will hold. Then, by Corollary \ref{cor-general}, the equality of \eqref{eq-comparison} holds for $i=1,2,\cdots,|B|$ on the graph equipped with the re-scaled weight.

To study more explicitly the rigidity of \eqref{eq-comparison}, we first consider the case that $\rho$ is constant.
\begin{cor}\label{cor-constant}Let $(G, m,w,B)$ be a connected weighted finite graph with boundary and suppose that there is a positive constant $\rho$,  such that for each $x\in B$ and $y\in\Omega$, $w_{xy}=\rho m_xm_y$. Then $\sigma_i=\mu_i$ for all $i=1,2,\cdots, |B|$ if and only if
\begin{equation}
\rho(V_\Omega-V_B)\leq \mu_2(\Omega).
\end{equation}
 In particular, if the induced graph on $\Omega$ is disconnected, then $\sigma_i=\mu_i$ for all $i=1,2,\cdots, |B|$ if and only if
\begin{equation}
V_B\geq V_\Omega.
\end{equation}
\end{cor}
Next, we consider the rigidity of \eqref{eq-comparison} for graphs with unit weight.
\begin{cor}\label{cor-unit}
Let $(G,B)$ be a connected finite graph with boundary equipped with the unit weight. Then, the equality of \eqref{eq-comparison} holds for all $i=1,2,\cdots,|B|$ if and only if all the following statements are true:
\begin{enumerate}
\item Each interior vertex is either adjacent to all boundary vertices or adjacent to no boundary vertex;
\item Let $\Omega_0$ be the set of interior vertices that are not adjacent to any boundary vertex and $\Omega_1=\Omega\setminus \Omega_0$, then $E(\Omega_0,\Omega_1)=|\Omega_0||\Omega_1|$ which means for any $x\in\Omega_0$ and $y\in\Omega_1$, $x$ and $y$ are adjacent;
\item $\mu_2(\Omega_1)\geq|\Omega_1|-|\Omega_0|-|B|$.
\end{enumerate}
\end{cor}

Moreover, by \eqref{eq-comparison}, a lower bound on $\mu_i$ will automatically give us a lower bound for $\sigma_i$.  For example, by the Lichnerowicz estimate for $\mu_2$ in \cite{BC,KKR} with respect to Bakry-\'Emery curvature, one has the following Lichnerowicz estimate for $\sigma_2$  directly.
\begin{cor}
Let $(G,m,w,B)$ be a connected weighted finite graph with boundary. Suppose that $(G,m,w)$ satisfy the Bakry-\'Emery curvature-dimension $\CD(K,n)$ with $K>0$ and $n>1$. Then,
\begin{equation}\label{eq-Lich-Steklov}
\sigma_2\geq \frac{nK}{n-1}.
\end{equation}
\end{cor}
In \cite{SY2}, we give a direct proof to \eqref{eq-Lich-Steklov} and discuss its rigidity.

Similarly, by the Licherowicz estimate for $\mu_2$ with respect to Ollivier curvature \cite{LLY}, one can have a Licherowicz estimate for $
\sigma_2$ with respect to Ollivier curvature. Here, the Ollivier curvature we used is the most general one given by M\"unch and Wojciechowski \cite{MW} recently.  Their definition is a natural extension of the definition by Lin-Lu-Yau \cite{LLY} on general weighted graphs. In \cite{MW}, M\"unch and Wojciechowski showed that the Ollivier curvature $\kappa(x,y)$ they defined  can be computed by the following formula:
\begin{equation}\label{eq-def-kappa}
\kappa(x,y)=\inf_{f\in \Lip(1),\nabla_{yx}f=1}\nabla_{xy}\Delta f
\end{equation}
for any two distinct vertices $x,y$, where
\begin{equation}
\nabla_{xy}f:=\frac{f(x)-f(y)}{d(x,y)}.
\end{equation}
By substituting an eigenfunction $f$ of $\mu_2$ into \eqref{eq-def-kappa}, one obtain the Lichnerowicz estimate:
 \begin{equation}
 \mu_2\geq \kappa
 \end{equation}
 for $\mu_2$ directly when the Ollivier curvature of $(G,m,w)$ has a positive lower bound $\kappa$. Combining this with \eqref{eq-comparison}, one has the following Lichnerowicz estimate for $\sigma_2$.
 \begin{cor}
Let $(G,m,w,B)$ be a connected weighted finite graph with boundary. Suppose that the Ollivier curvature of $(G,m,w)$ is not less than a positive constant $\kappa$. Then,
\begin{equation}\label{eq-Lich-Ollivier}
\sigma_2\geq \kappa.
\end{equation}
\end{cor}

Furthermore, by using the lower bound of $\mu_2$ by Fiedler \cite{FI}, we have the following lower bounds for $\sigma_2$.
\begin{thm}
Let $G$ be a connected finite graph equipped with the unit weight. Then,
\begin{equation}
\sigma_2\geq 2e(G)\left(1-\cos\frac\pi{|V(G)|}\right)\geq 2v(G)\left(1-\cos\frac\pi{|V(G)|}\right)
\end{equation}
where $e(G)$ is edge connectivity of $G$ and $v(G)$ is the vertex connectivity of $G$. That is, the least number of edges and least number of vertices in $G$ required to be deleted to make $G$ becoming disconnected respectively.
\end{thm}

Finally, by using the lower bounds of
$\mu_i$'s obtained by Friedman in \cite{FR}, one has the following lower bounds for Steklov eigenvalues. For the definition of a star, see \cite[P.1]{FR}.
\begin{thm}\label{thm-F}
Let $(G,B)$ be a connected finite graph with boundary equipped with the unit weight and $i\geq 2$. Then,
\begin{enumerate}
\item when $i\not\big |\ |V(G)|$,
\begin{equation}\label{eq-not-divide}
\sigma_i\geq 2-2\cos\frac{\pi}{2k+1}
\end{equation}
where $k=\left\lfloor \frac{|V(G)|}{i}\right\rfloor$. When $|V(G)|\equiv 1(\Mod\ i)$, the equality of \eqref{eq-not-divide} holds if and only if $k=1$, i.e. $|V(G)|=i+1$, and $(G, B)$ is a star of degree $i$ with each arm of length $1$, and with the end points of each arm the boundary vertices.
\item when $i\ \big|\ |V(G)|$,
\begin{equation}\label{eq-divide}
\sigma_i>\mathcal P(k,\lambda_i).
\end{equation}
 Here $k=\frac{|V(G)|}{i}$, $\lambda_i$ is the largest eigenvalue of the path $P_i$ on $i$ vertices equipped with the unit weight, and $\mathcal P(k,\lambda)$ is the first Dirichlet eigenvalue of $(P_{k+1}, m, w, B)$ where $P_{k+1}$ is a path with vertices $0,1,2,\cdots, k$ and $B=\{0\}$, and moreover $m_j=1$ for $j=0,1,\cdots, k$, $w_{12}=w_{23}=\cdots=w_{k-1,k}=1$ and $w_{01}=\lambda$.
\end{enumerate}
\end{thm}

When $|V(G)|\equiv s\ (\Mod\ i)$ with $2\leq s\leq i-1$, we lack of  rigidity for \eqref{eq-not-divide} because we lack of rigidity  for the corresponding estimate for $\mu_i$ in \cite{FR}. The same as in \cite{FR}, in this case, the equality of \eqref{eq-not-divide} at least holds for two different graphs. For example, when $|V(G)|=5$, $i=3$, the equality of \eqref{eq-not-divide} holds for a star of degree $4$ with all arms of length $1$, and with  the end points of three arms as the boundary vertices or with all the four end points of the four arms as boundary vertices. It also holds for a star of degree $3$ with two arms of length $1$ and one arm of length $2$ and with the the end points of the two arms of length $1$ and the middle point the arm of length $2$ as the three boundary vertices.

The rest of the paper is organized as follows. 

\section{Comparison of $\sigma_i$ and $\mu_i$}
Although \eqref{eq-comparison} comes from an almost obvious observation via Courant's min-max principle as mentioned in the last section, we will present a detailed proof below for convenience of handling the rigidity of \eqref{eq-comparison}.
\begin{proof}[Proof of Theorem \ref{thm-comparison}] Let $f_1=1, f_2,\cdots, f_{|B|}\in \R^B$ be the eigenfunctions of $\sigma_1=0,\sigma_2,\cdots,\sigma_{|B|}$ respectively, such that
\begin{equation}
  \vv<f_i,f_j>_B=0
\end{equation}
when $i\neq j$. Moreover, let $u_1=1,u_2,\cdots,u_{|V|}\in \R^V$ be the eigenfunctions of $\mu_1=0,\mu_2,\cdots,\mu_{|V|}$ respectively, such that
\begin{equation}
\vv<u_i,u_j>=0
\end{equation}
when $i\neq j$. For $i=2,3,\cdots, |B|$, let
$$v_i=c_1u_{f_1}+c_2u_{f_2}+\cdots c_i u_{f_i}$$
with $c_1,c_2,\cdots,c_i$ not all zero, be such that
\begin{equation}\label{eq-perp}
\vv<v_i,u_j>=0\ {\rm for}\ j=1,2,\cdots, i-1.
\end{equation}
This can be done because \eqref{eq-perp} is a homogeneous linear system with $i-1$ equations and $i$ unknowns $c_1,c_2,\cdots, c_{i}$ which  certainly has nonzero solutions. Then
\begin{equation}\label{eq-min-max}
\begin{split}
\mu_i\leq \frac{\vv<dv_i,dv_i>}{\vv<v_i,v_i>}\leq \frac{\vv<dv_i,dv_i>}{\vv<v_i,v_i>_B}\leq\sigma_i.
\end{split}
\end{equation}
It is clear that the equality $\sigma_i=\mu_i$ holds only when $v_i|_\Omega=\Delta v_i|_\Omega=0$ and $v_i$ is simultaneously an eigenfunction of $\mu_i$ and $\sigma_i$.

We next come to the rigidity part. If the equalities of \eqref{eq-comparison} holds for $i=1,2,\cdots, |B|$, we first claim that there is a sequence $\tilde v_1=1,\tilde v_2,\cdots,\tilde v_{|B|}$ of nonzero functions on $V$ such that
\begin{enumerate}
\item[(i)] $\tilde v_i|_{\Omega}=\Delta\tilde v_i|_{\Omega}=0$ for $i=2,\cdots,|B|$;
\item[(ii)] $\frac{\p \tilde v_i}{\p n}=\sigma_i \tilde v_i$ for $i=1,2,\cdots,|B|$;
\item[(iii)] $\Delta \tilde v_i=\mu_i \tilde v_i$ for $i=1,2,\cdots, |B|$;
\item[(iv)] $\vv<\tilde v_i,\tilde v_j>_B=\vv<\tilde v_i,\tilde v_j>=0$ when $1\leq j<i\leq |B|$.
\end{enumerate}
We now construct the sequence $\tilde v_1,\tilde v_2,\cdots,\tilde v_{|B|}$ by induction. For $i\geq 2$, suppose $\tilde v_1,\tilde v_2,\cdots,\tilde v_{i-1}$ satisfying (i),(ii),(iii) and (iv) has been constructed.
Let
$$\tilde v_i=c_1u_{f_1}+c_2u_{f_2}+\cdots+c_{i-1}u_{f_{i-1}}+c_i u_{f_i}$$
with $c_1,c_2,\cdots,c_i$ not all zero, be such that
\begin{equation}
\vv<\tilde v_i, \tilde v_j>=0
\end{equation}
for all $j=1,2,\cdots, i-1$. This can be done because of the same reason as before. Then, by  replacing the function $v_i$ by $\tilde v_i$ in \eqref{eq-min-max}, and noting that $\mu_i=\sigma_i$, we know that $\tilde v_i$ must satisfy (i),(ii),(iii) and (iv).

Note that for any $f\in \R^B$ with
\begin{equation}\label{eq-perp-f}
\sum_{x\in B}f(x)m_x=\vv<f,1>_B=0,
\end{equation}
 one has
\begin{equation}
f=\sum_{i=2}^{|B|}c_i\tilde v_i|_B
\end{equation}
for some $c_2,\cdots,c_{|B|}$. Then
\begin{equation}
u_f=\sum_{i=2}^{|B|}c_i\tilde v_i
\end{equation}
because $\Delta \tilde v_i|_\Omega=0$. Moreover, by that $\tilde v_i|_\Omega=0 $ for $i=2,3,\cdots,|B|$, $u_f(y)=0$ for all $y\in\Omega$. So, for any $y\in \Omega$, we have
\begin{equation}\label{eq-perp-f-2}
\begin{split}
0=m_y\Delta u_f(y)=\sum_{x\in B}f(x){w_{xy}}.
\end{split}
\end{equation}
Comparing this to \eqref{eq-perp-f}, we know that $w_{xy}=\kappa_y\cdot m_x$ for any $x\in B$ and $y\in \Omega$ for some nonnegative function $\kappa\in \R^\Omega$. Let $\rho_y=\frac{\kappa_y}{m_y}$ for $y\in\Omega$. Then, we get (1) of the theorem.

Conversely, when $w_{xy}=\rho_ym_xm_y$ for any $x\in B$ and $y\in\Omega$, for any nonzero $f\in \R^B$ with $\vv<f,1>_B=0$, one has $u_f(y)=0$ for any $y\in\Omega$. So,
\begin{equation}
 \Lambda f(x)=\frac{\p u_f}{\p n}(x)=\frac1{m_x}\sum_{y\in\Omega}f(x)w_{xy}=\Deg\cdot f(x)
\end{equation}
for any $x\in B$, and
\begin{equation}
-\Delta u_f=\Deg\cdot u_f
\end{equation}
This implies that $\sigma_2=\sigma_3=\cdots=\sigma_{|B|}=\Deg$, and equality of \eqref{eq-comparison} holds for any $i=1,2,\cdots, |B|$ if and only if
\begin{equation}\label{eq-v}
\vv<dv,dv>\geq \Deg\cdot \vv<v,v>.
\end{equation}
for any nonzero $v\in \R^V$ with $\vv<v,u_{f_i}>=0$ for $i=1,2,\cdots,|B|$.

Note that $u_{f_i}|_\Omega=0$ for $i=2,3,\cdots, |B|$, so
\begin{equation}
\vv<v,u_{f_i}>_B=\vv<v,u_{f_i}>=0
\end{equation}
for $i=2,3,\cdots,|B|$. This implies that $v|_B$ must be constant.

When $v|_B\equiv0$, let $u=v|_\Omega$. Then $\vv<u,1>_\Omega=\vv<v,1>=0$ and
\begin{equation}
\vv<dv,dv>=\vv<du,du>_\Omega+V_B\vv<\rho u,u>_\Omega.
\end{equation}
So, by \eqref{eq-v}, we have
\begin{equation}\label{eq-u-1}
 \vv<du,du>_\Omega-\Deg\vv<u,u>_\Omega+V_B\vv<\rho u,u>_\Omega\geq 0
\end{equation}
for any $u\in \R^\Omega$ with $\vv<u,1>_\Omega=0$.

When $v|_B$ is nonzero, without loss of generality, we assume that $v|_B\equiv 1$. Let $u\in \R^\Omega$ be such that $u=v|_\Omega+\frac{V_B}{V_\Omega}$. Then
\begin{equation}
\vv<u,1>_\Omega=\vv<v+\frac{V_B}{V_\Omega},1>_\Omega=\vv<v,1>-\vv<v,1>_B+V_B=0.
\end{equation}
Moreover,
\begin{equation}\label{eq-v-1}
\begin{split}
\vv<dv,dv>=&\vv<du,du>_\Omega+V_B\vv<\rho\left(u-\frac{V_G}{V_\Omega}\right),u-\frac{V_G}{V_\Omega}>_\Omega\\
=&\vv<du,du>_\Omega+V_B\vv<\rho u,u>_\Omega-\frac{2V_BV_G}{V_\Omega}\vv<\rho,u>_\Omega+\frac{\Deg V_BV_G^2}{V_\Omega^2}.\\
\end{split}
\end{equation}
On the other hand,
\begin{equation}\label{eq-v-2}
\vv<v,v>=\vv<v,v>_\Omega+\vv<v,v>_B=\vv<u,u>_\Omega+\frac{V_BV_G}{V_\Omega}.
\end{equation}
Substituting \eqref{eq-v-1} and \eqref{eq-v-2} into \eqref{eq-v}, we have
\begin{equation}\label{eq-u-2}
\vv<du,du>_\Omega-\Deg\vv<u,u>_\Omega+V_B\vv<\rho u,u>_\Omega-\frac{2V_BV_G}{V_\Omega}\vv<\rho,u>_\Omega+\frac{\Deg V_B^2 V_G}{V_\Omega^2}\geq0
\end{equation}
for any $u\in \R^\Omega$ with $\vv<u,1>_\Omega=0$. Note that for any constant $\lambda$, $\vv<\lambda u,1>_\Omega=0$, replacing $u$ by  $\lambda u$ in \eqref{eq-u-2}, one has
\begin{equation}
\left(\vv<du,du>_\Omega-\Deg\vv<u,u>_\Omega+V_B\vv<\rho u,u>_\Omega\right)\lambda^2-\frac{2V_BV_G}{V_\Omega}\vv<\rho,u>_\Omega\lambda +\frac{\Deg V_B^2 V_G}{V_\Omega^2}\geq0
\end{equation}
for all $\lambda\in \R$. This is equivalent to
\begin{equation}\label{eq-u-3}
\left(\frac{2V_BV_G}{V_\Omega}\vv<\rho,u>_\Omega \right)^2\leq 4\left(\vv<du,du>_\Omega-\Deg\vv<u,u>_\Omega+V_B\vv<\rho u,u>_\Omega\right)\frac{\Deg V_B^2 V_G}{V_\Omega^2}.
\end{equation}
Conversely, it is not hard to see that \eqref{eq-u-3} implies \eqref{eq-u-2}. Simplifying \eqref{eq-u-3}, we get
\begin{equation}\label{eq-u-4}
\vv<du,du>_\Omega-\Deg\vv<u,u>_\Omega+V_B\vv<\rho u,u>_\Omega-\frac{V_G}{\Deg}\vv<\rho,u>_\Omega^2\geq 0
\end{equation}
for any $u\in\R^\Omega$ with $\vv<u,1>_\Omega=0$. Because \eqref{eq-u-4} is stronger than \eqref{eq-u-1}, we only require \eqref{eq-u-4}.

Conversely, it is not hard to see that \eqref{eq-u-4} implies \eqref{eq-v} because \eqref{eq-u-3} implies \eqref{eq-u-2}. This completes the proof of the theorem.

\end{proof}
We next come to prove Corollary \ref{cor-general}, a sufficient condition for \eqref{eq-comparison} to hold on general weighted graphs.
\begin{proof}[Proof of  Corollary \ref{cor-general}] Note that, for any $u\in \R^\Omega$ with $\vv<u,1>_\Omega=0$,
\begin{equation}
\begin{split}
&\vv<du,du>_\Omega-\Deg\vv<u,u>_\Omega+V_B\vv<\rho u,u>_\Omega-\frac{V_G}{\Deg}\vv<\rho,u>_\Omega^2\\
\geq&\left(\mu_2(\Omega)-\Deg+V_B\rho_{\min}-\frac{V_G}{\Deg}\vv<\rho,\rho>_\Omega\right)\vv<u,u>_\Omega\\
\geq& 0.
\end{split}
\end{equation}
So, by Theorem \ref{thm-comparison}, we get the conclusion.
\end{proof}
We next come to prove Corollary \ref{cor-constant},  the rigidity of \eqref{eq-comparison} for the case that $\rho$ is constant.
\begin{proof}[Proof of Corollary \ref{cor-constant}] Note that, for any $u\in\R^\Omega$ with $\vv<u,1>_\Omega=0$,
\begin{equation}
\begin{split}
&\vv<du,du>_\Omega-\Deg\vv<u,u>_\Omega+V_B\vv<\rho u,u>_\Omega-\frac{V_G}{\Deg}\vv<\rho,u>_\Omega^2\\
=&\vv<du,du>_\Omega-\Deg\vv<u,u>_\Omega+\rho V_B\vv<u,u>_\Omega\\
=&\vv<du,du>_\Omega-\rho(V_\Omega-V_B)\vv<u,u>_\Omega
\end{split}
\end{equation}
since $\Deg=\rho V_\Omega$ in this case. So, \eqref{eq-nonnegative} holds for any $\vv<u,1>_\Omega=0$ if  and only if
\begin{equation}
\mu_2(\Omega)\geq \rho(V_\Omega-V_B).
\end{equation}
This completes the proof of the corollary.
\end{proof}

We next prove Corollary \ref{cor-unit}, the rigidity of \eqref{eq-comparison} for graphs with unit weight.
\begin{proof}[Proof of Corollary \ref{cor-unit}]
 By Theorem \ref{thm-comparison}, it is clear that (i) is true. Moreover, it is clear that $\Omega_1\neq \emptyset$. If $\Omega_0$ is empty, it reduces to case that $\rho\equiv 1$ and by Corollary \ref{cor-constant}, we know that the conclusion is true. So, in the following, we assume that $\Omega_0\neq\emptyset$.

 Because the graph is of unit weight, we know that $\rho_x=0$ when $x\in \Omega_0$ and $\rho_y=1$ when $y\in \Omega_1$. So $\Deg=|\Omega_1|$.

 For each $u\in \R^\Omega$, let $\bar u_0=\vv<u,1>_{\Omega_0}/|\Omega_0|$ and $\bar u_1=\vv<u,1>_{\Omega_1}/|\Omega_1|$, $v_0\in \R^{\Omega_0}$ with $v_0=u|_{\Omega_0}-\bar u_0$ and $v_1\in \R^{\Omega_1}$ with $v_1=u|_{\Omega_1}-\bar u_1$. Then
 \begin{equation}\label{eq-nonnegative-2}
 \begin{split}
 &\vv<du,du>_\Omega-\Deg\vv<u,u>_\Omega+V_B\vv<\rho u,u>_\Omega-\frac{V_G}{\Deg}\vv<\rho,u>_\Omega^2\\
 =&\vv<dv_0,dv_0>_{\Omega_0}-|\Omega_1|\vv<v_0,v_0>_{\Omega_0}+\sum_{x\in\Omega_0}v_0^2(x)\deg_{\Omega_1}(x)+2(\bar u_0-\bar u_1)\sum_{x\in\Omega_0}v_0(x)\deg_{\Omega_1}(x)\\
 &-2\sum_{x\in\Omega_0}\sum_{y\in\Omega_1}v_0(x)v_1(y)w_{xy}+\vv<dv_1,dv_1>_{\Omega_1}-(|\Omega_1|-|B|)\vv<v_1,v_1>_{\Omega_1}\\
 &+\sum_{y\in\Omega_1}v_1^2(y)\deg_{\Omega_0}(y)+2(\bar u_1-\bar u_0)\sum_{y\in\Omega_1}v_1(y)\deg_{\Omega_0}(y)\\
 &+(\bar u_0-\bar u_1)^2|E(\Omega_0,\Omega_1)|-\bar u_1^2(2|\Omega_1|+|\Omega_0|)|\Omega_1|-\bar u_0^2|\Omega_0||\Omega_1|,
 \end{split}
 \end{equation}
 where $$\deg_{\Omega_1}(x)=\sum_{y\in\Omega_1}w_{xy}$$ for any $x\in\Omega_0$ and $$\deg_{\Omega_0}(y)=\sum_{x\in\Omega_0}w_{xy}$$ for any $y\in\Omega_1$. So \eqref{eq-nonnegative} holds for any $u\in \R^\Omega$ with $\vv<u,1>_\Omega=0$ if and only if \eqref{eq-nonnegative-2} is nonnegative for any $v_0\in \R^{\Omega_0}$, $v_1\in \R^{\Omega_1}$, $\bar u_0,\bar u_1\in \R$ with $\vv<v_0,1>_{\Omega_0}=\vv<v_1,1>_{\Omega_1}=0$ and $\bar u_0|\Omega_0|+\bar u_1|\Omega_1|=0$.

 By setting $v_0=v_1=0$ and $\bar u_0=-|\Omega_1|$ and $\bar u_1=|\Omega_0|$ in \eqref{eq-nonnegative-2}, we know that
 \begin{equation}
 |E(\Omega_0,\Omega_1)|\geq|\Omega_0||\Omega_1|
 \end{equation}
 if \eqref{eq-nonnegative} holds. Because we always assume that the graph is simple, so $$E(\Omega_0,\Omega_1)=|\Omega_0||\Omega_1|.$$
 Moreover, $w_{xy}=1$,$\deg_{\Omega_1}(x)=|\Omega_1|$ and $\deg_{\Omega_0}(y)=|\Omega_0|$ for all $x\in \Omega_0$ and $y\in \Omega_1$. Substituting all these into \eqref{eq-nonnegative-2}, we know that the nonnegativity of \eqref{eq-nonnegative-2} reduces to that
 \begin{equation}
 \vv<dv_0,dv_0>_{\Omega_0}+\vv<dv_1,dv_1>_{\Omega_1}-(|\Omega_1|-|B|-|\Omega_0|)\vv<v_1,v_1>_{\Omega_1}\geq 0
 \end{equation}
 for any $v_0\in \R^{\Omega_0}$ and $v_1\in\R^{\Omega_1}$ with $\vv<v_0,1>_{\Omega_0}=\vv<v_1,1>_{\Omega_1}=0$. This is clearly equivalent to
 \begin{equation}
 \mu_2(\Omega_1)\geq |\Omega_1|-|\Omega_0|-|B|.
 \end{equation}
 This completes the proof of the corollary.
\end{proof}

Finally, we come to prove Theorem \ref{thm-F}.
\begin{proof}[Proof of Theorem \ref{thm-F}]
(1) By \cite[Theorem 1.2]{FR},
 \begin{equation}\label{eq-mu-1}
 \mu_i\geq 2-2\cos\frac{\pi}{2k+1}.
 \end{equation}
 Combining this with \eqref{eq-comparison}, we get \eqref{eq-not-divide}.

 When the equality of \eqref{eq-not-divide} holds, we know that the equality of \eqref{eq-mu-1} holds. When $|V(G)|=ik+1$, by the rigidity in \cite[Theorem 1.2]{FR},  $G$ is a star of degree $i$ with each arm of length $k$. Denote the center of the star as $o$, and the vertices of the $j^{\rm th}$ arm as $v_{j1},v_{j2},\cdots, v_{jk}$ for $j=1,2,\cdots, k$ (with $v_{jk}$ the end point of the arm). Then,
\begin{equation}
\mu_2=\mu_3=\cdots=\mu_i
\end{equation}
and the corresponding eigenspace is generated by
\begin{equation}\label{eq-eigen-1}
f_j(x)=\left\{\begin{array}{rl}0&x=o\\
f(s)&x=v_{1s}\\
-f(s)&x=v_{js}\\
0&{\rm otherwise}
\end{array}\right.
\end{equation}
for $j=2,3,\cdots,i$. Here $f$ is the first Dirichlet eigenfunction of the path on $k+1$ vertices: $0,1,2,\cdots, k$ with $0$ the boundary vertex equipped with the unit weight. Without loss of generality, we can assume that $f(i)>0$ when $i=1,2,\cdots,k$.

Moreover, by theorem \ref{thm-comparison}, because $\sigma_i=\mu_i$ with $i\geq 2$, there must be an eigenfunction $v_i$ of $\mu_i$ such that $v_i|_\Omega=0$.  Because the eigenspace of $\mu_i$ is generated by the functions listed in \eqref{eq-eigen-1}, we know that every eigenfunction of $\mu_i$ must be positive on at least one arm (except the center).  Combining this with the fact that we require two boundary vertices not adjacent to each other, one has $k=1$ and the conclusion follows.

(2) By \cite[Theorem 1.4]{FR},
\begin{equation}\label{eq-mu-2}
\mu_i\geq \mathcal P(k,\lambda_i).
\end{equation}
Combining this with \eqref{eq-comparison}, one has
\begin{equation}\label{eq-sigma}
\sigma_i\geq \mathcal P(k,\lambda_i).
\end{equation}
If the equality holds, then by \eqref{eq-comparison}, the equality of \eqref{eq-mu-2} holds. Then, when $k>1$ or $|V(G)|$ is even, by the rigidity in \cite[Theorem 1.4]{FR}, we know that $G$ is a comb of degree $i$ with each tooth of length $k-1$ (For the definition of a comb, see \cite[P. 2]{FR}). Let the path on $v_{11},v_{21},\cdots, v_{i1}$ be the base of the comb and the path on $v_{j1},v_{j2},,\cdots, v_{jk}$ be the tooth on $v_{j1}$ for $j=1,2,\cdots, i$. Let $g$ be the first Dirichlet eigenfunction of $(P_{k+1}, m,w,B)$ which is positive except on the boundary vertex and $h$ be a top eigenfunction for the base of the comb.  Then, $\mu_i$ is a simple eigenvalue with eigenfunction
\begin{equation}\label{eq-eigen-2}
f(v_{rs})=g(s)h(r)
\end{equation}
for $r=1,2,\cdots,i$ and $s=1,2,\cdots, k$. Moreover, by Theorem \ref{thm-comparison}, $f$ must be vanished on $\Omega$. This implies that $\Omega$ is empty because $f$ is everywhere non-vanished which violates that $\Omega$ is not empty. Hence, the equality of \eqref{eq-sigma} can not hold.

When $k=1$ and $|V(G)|$ is odd, then $G$ is a path or a cycle. Because both of the top eigenfunction of a path and a cycle will not be vanished on any vertex, the same reason as before implies that the equality \eqref{eq-sigma} can not hold.
Thus, we complete the proof of the conclusion.
\end{proof}
\begin{rem}
In \cite{FR}, the author claim that the equality of \eqref{eq-mu-2} holds only when the graph is a comb with $i$ teeth and with each tooth of length $k-1$. However, this is not true when $i=|V(G)|$ and $|V(G)|$ is odd. In fact, it is not hard to see that in this case $G$ can be either a path or a cycle because the top eigenvalues of the path and the cycle with an odd number of vertices are the same (see \cite[P. 9]{BR}). By the argument in \cite{FR}, it is not hard to see that these are the only two graphs that the equality of \eqref{eq-mu-2} holds in this case.
\end{rem}

\end{document}